\documentclass[11pt]{article}%
\usepackage{amssymb}
\usepackage{amsmath}
\usepackage[spanish]{babel}
\usepackage{amsfonts}
\usepackage{graphicx}%
\newtheorem{theorem}{Teorema}

\newtheorem{definition}[theorem]{Definici\'{o}n}
\newtheorem{example}[theorem]{Ejemplo}
\newtheorem{lemma}[theorem]{Lema}
\newtheorem{proposition}[theorem]{Proposici\'{o}n}
\newenvironment{proof}
{\trivlist
\parindent=0cm \leftskip=1em \rightskip=1em
\hangindent=0em \small
\item[\hskip.2em Demostraci\'{o}n.--]
\hskip2pt}
{\hfill QED. \endtrivlist}
\begin{document}

\title{L\'{\i}mites directos de prolongaciones \\de algebroides de Lie}
\author{Patrick CABAU\\{\small {Lyc\'{e}e Pierre de Fermat, BP 7013, 31068 Toulouse cedex, France}}\\{\small {e-mail: Patrick.Cabau@ac-toulouse.fr} }}
\maketitle

{\small \textbf{Abstract.--} We prove that direct limits of
Lie algebroids and their prolongations can be endowed with structures of
convenient spaces.}

{\small \textbf{Resumen.--} Probamos que se puede definir estructuras de
espacios convenientes} {\small sobre l\'{\i}mites directos de algebroides de
Lie y sus prolongaciones.}

{\small \textbf{R\'{e}sum\'{e}.--} On \'{e}tablit que l'on peut munir les
limites directes d'alg\'{e}bro\"{\i}des de Lie et de leurs prolongements de
structures d'espaces }$c^{\infty}-${\small complets.}


{\small \textbf{Palabras claves:} L\'{\i}mite directo; algebroide de Lie;
prolongaci\'{o}n de algebroide de Lie; c\'{a}lculo diferencial conveniente}

{\small \textbf{Mots-clefs :} Limite directe\,; alg\'{e}bro\"{\i}de de Lie\,;
prolongement d'alg\'{e}bro\"{ide} de Lie\,; calcul diff\'{e}rentiel au sens de
Kriegel et Michor}


\section{Introduction}

Los algebroides de Lie definidos por J. Pradines en \cite{Pra} constituyen una
generalizaci\'{o}n natural de las \'{a}lgebras de Lie y de los fibrados
tangentes a una variedad.

En los \'{u}ltimos a\~{n}os, esta noci\'{o}n se ha revelado fecunda en
Mec\'{a}nica, Geometr\'{\i}a simpl\'{e}ctica y teor\'{\i}a del Control
\'{o}ptimo. En \cite{Wei}, Weinstein plantea el problema de un posible
desarrollo de un formalismo geom\'{e}trico sobre algebroides de Lie similar al
formalismo de Klein en Mec\'{a}nica Lagrangiana. En \cite{Mar1}, Mart\'{\i}nez
da una respuesta positiva a este problema usando la noci\'{o}n de
prolongaci\'{o}n de un algebroide de Lie introducida por Higgins y Mackenzie
en \cite{HigMac}.

En este art\'{\i}culo nos interesamos a los l\'{\i}mites directos de tales
estructuras y obtenemos los resultados dados en los teoremas \ref{T_ConvenienLieAlgebroid} y
\ref{T_ConvenientLieAlgebroidExtension} de la \'{u}ltima secci\'{o}n: se puede
definir sobre estos l\'{\i}mites una estructura conveniente.

En la secci\'{o}n \ref{_CalculoDiferencialConvenient} recordamos el formalismo
conveniente desarrollado en \cite{KriMic}. Desarrollamos la noci\'{o}n de
l\'{\i}mite directo de diferentes estructuras (espacios vectoriales
topologicos, variedades, fibrados vectoriales) en las secciones
\ref{_LimitesDirectosEVT} y \ref{_LimitesDirectosVariedades}. Entonces, obtenemos estructuras convenientes sobre estos l\'{\i}mites (Proposici\'{o}n
\ref{T_DirectLimitSequenceFiniteDimensionalParacompactManifolds} y
Proposici\'{o}n \ref{P_StructureOnDirectLimitVectorBundles}). En la
secci\'{o}n \ref{_AlgebroidesLie} recordamos las nociones de algebroides de Lie y sus prolongaciones. En la \'{u}ltima secci\'{o}n probamos que se puede definir sobre l\'{\i}mites directos de tales objetos estructuras convenientes y damos el ejemplo del oscilador arm\'{o}nico conveniente.

\section{\label{_CalculoDiferencialConvenient}C\'{a}lculo diferencial
conveniente}

Con el fin de equipar un espacio vectorial topol\'{o}gico de Haussdorf
localmente convexo (e.v.t.l.c.) $E$ con una estructura diferencial, como introducida por Fr\"{o}licher, Kriegl y Michor, se utiliza la noci\'{o}n de
curva diferenciable $c:\mathbb{R}\rightarrow E$ de clase $C^{\infty}$ que no
plantea ning\'{u}n problema.

La propiedad clave es la de $c^{\infty}$-completitud.

\begin{definition}
\label{D_ConvenientVectorSpace} Se dice que un e.v.t.l.c. $E$ es $c^{\infty}$-completo o \textit{conveniente} si se cumple la propiedad siguiente:
Dado $B\subset E$ un conjunto cerrado, acotado y absolutamente convexo, el espacio lineal $E_{B}$ generado por $B$ es un espacio de Banach.
\end{definition}

Un e.v.t.l.c. conveniente es Mackey completo (cf. \cite{KriMic} Teorema 2.14).

La $c^{\infty}$-topolog\'{\i}a de un e.v.t.l.c. es la topolog\'{\i}a final
inducida por la familia de las curvas $C^{\infty}$ $\mathbb{R}\rightarrow E$;
se denotar\'{a} por $c^{\infty}E$. A sus conjuntos abiertos los llamaremos
$c^{\infty}$-abiertos.

N\'{o}tese que la $c^{\infty}$-topolog\'{\i}a es en general m\'{a}s fina que
la topolog\'{\i}a original. En los espacios de Fr\'{e}chet, esta
topolog\'{\i}a coincide con la topolog\'{\i}a del e.v.t.l.c.

En general, $c^{\infty}E$ no es un espacio vectorial localmente convexo.

Sean $E$ y $F$ dos espacios convenientes y sea $U\subset E$ un $c^{\infty}%
$-abierto. Se dice que una aplicaci\'{o}n $f:E\supset U\rightarrow F$ es
$c^{\infty}$ si $f\circ c\in C^{\infty}\left(  \mathbb{R},F\right)  $ para
cada $c\in C^{\infty}\left(  \mathbb{R},U\right)$.

Adem\'{a}s, se puede definir una estructura de espacio vectorial $c^{\infty}%
$-completo sobre el espacio $C^{\infty}\left(  U,F\right)$ (cf.
\cite{KriMic}, 2.3 (5)).

\begin{proposition}
Los l\'{\i}mites, las sumas directas y los l\'{\i}mites directos estrictos de
espacios convenientes son espacios convenientes.
\end{proposition}

Las nociones de variedad conveniente (v\'{e}ase \cite{KriMic}, 27.) y de
fibrado vectorial conveniente (v\'{e}ase \cite{KriMic}, 29.) se definen de
manera natural.

\section{\label{_LimitesDirectosEVT}L\'{\i}mites directos de espacios
vectoriales topol\'{o}gicos}

\bigskip En esta secci\'{o}n vamos a referirnos a \cite{Bou}, \cite{Glo1} y
\cite{Glo2}.

Sea $\left(  I,\leq\right)  $ un conjunto direccionado. Un \textit{sistema
directo} en una categor\'{\i}a $\mathbb{A}$ es un par $\mathcal{S}=\left(
X_{i},\varepsilon_{i}^{j}\right)  _{i\in I,\ j\in I,\ i\leq j}$ donde $X_{i}$
es un objeto de la categor\'{\i}a y los $\varepsilon_{i}^{j}:X_{i}\longrightarrow
X_{j}$ son homomorfismos (\textit{bonding maps}) que cumplen las propiedades siguientes:

\begin{enumerate}
\item $\forall i \in I,\varepsilon_{i}^{i}=\operatorname{Id}_{X_{i}}$

\item $\forall\left(  i,j,k\right)  \in I^{3},i\leq j\leq k,\ \varepsilon
_{j}^{k}\circ\varepsilon_{i}^{j}=\varepsilon_{i}^{k}$.
\end{enumerate}

Un \textit{cono} sobre $\mathcal{S}$ es un par $\left(  X,\varepsilon
_{i}\right)  _{i\in I}$ donde $X\in\operatorname*{ob}\mathbb{A}$,
$\varepsilon_{i}:X_{i}\longrightarrow X$ y los homomorfismos $\varepsilon_{i}^{j}:X_{i}\longrightarrow
X_{j}$ verifican las relaciones
$\varepsilon_{j}\circ\varepsilon_{i}^{j}=\varepsilon_{i}$ para $i\leq j$.

Un cono $\left(  X,\varepsilon_{i}\right)  _{i\in I}$ es un l\'{\i}mite
directo de $\mathcal{S}$ si para cualquier cono $\left(  Y,\theta_{i}\right)
_{i\in I}$ sobre $\mathcal{S}$ existe un \'{u}nico homomorfismo $\psi
:X\longrightarrow Y$ tal que $\psi\circ\varepsilon_{i}=\theta_{i}$. En tal
caso, escribiremos $X=\underrightarrow{\lim}\mathcal{S}$ o $X=\underrightarrow
{\lim}X_{i}$.

Cuando $I=\mathbb{N}$ dotado de la relaci\'{o}n de orden en los n\'{u}meros
naturales, los sistemas directos numerables se llaman \textit{sucesiones
directas}.

\subsection{L\'{\i}mites directos de conjuntos}

Sea $\mathcal{S}=\left(  X_{i},\varepsilon_{i}^{j}\right)  _{i\in I,\ j\in
I,\ i\leq j}$ un sistema directo de conjuntos ($\mathbb{A}=\mathbb{SET}$).

Sea $\mathcal{U}=\coprod\limits_{i\in I}X_{i}=\left\{  \left(  x,i\right)
:x\in X_{i}\right\}  $ la uni\'{o}n disjunta de los conjuntos $X_{i}$ con la
inclusi\'{o}n can\'{o}nica $%
\begin{array}
[c]{cccc}%
\iota_{i}: & X_{i} & \longrightarrow & \mathcal{U}\\
& x & \mapsto & \left(  x,i\right).
\end{array}$.
Se define una relaci\'{o}n de equivalencia $\sim$ sobre $\mathcal{U}$ de la
manera siguiente:\\
 $\iota_{i}\left(  x\right)  \sim\iota_{j}\left(  y\right)  $
si existe $k\in I$ : $i\leq k$ y $j\leq k$ tales que $\varepsilon_{i}%
^{k}\left(  x\right)  =\varepsilon_{j}^{k}\left(  y\right)$.\\
Tenemos el conjunto cociente $X=\mathcal{U}/\sim$ y la aplicaci\'{o}n $\varepsilon
_{i}:\pi\circ\iota_{i}$ donde $\pi:\mathcal{U}\longrightarrow\mathcal{U}/\sim$
es la aplicaci\'{o}n can\'{o}nica suprayectiva.

Entonces, $\left(  X,\varepsilon_{i}\right)  $ es el l\'{\i}mite directo
$\mathcal{S}$ en la categor\'{\i}a $\mathbb{SET}$.

Si cada $\varepsilon_{i}^{j}$ es inyectiva entonces $\varepsilon_{i}$ es
inyectiva. Por supuesto $S$ es equivalente al sistema directo de los
subconjuntos $\varepsilon_{i}\left(  X_{i}\right)  \subset X$, con las
inclusiones can\'{o}nicas.

\subsection{L\'{\i}mites directos de espacios top\'{o}logicos}

Si $\mathcal{S}=\left(  X_{i},\varepsilon_{i}^{j}\right)  _{i\in I,\ j\in
I,\ i\leq j}$ es un sistema directo de espacios topol\'{o}gicos $X_{i}$ donde
$\varepsilon_{i}^{j}$ son aplicaciones continuas, entonces el l\'{\i}mite
directo $\left(  X,\varepsilon_{i}\right)  _{i\in I}$ en la categor\'{\i}a de conjuntos coincide con el l\'{\i}mite directo en la categor\'{\i}a $\mathbb{TOP}$ de los espacios
topol\'{o}gicos si $X$ tiene la \textit{DL-topolog\'{\i}a}, i.e. la topolog\'{\i}a
m\'{a}s fina para la cual todas las aplicaciones $\varepsilon_{i}$ son
continuas. Entonces $O\subset X$ es abierto si y s\'{o}lo si $\varepsilon
_{i}^{-1}\left(  O\right)  $ es abierto en $X_{i}$ para cada $i\in I$.

$\mathcal{S}$ es \textit{estricto} si cada $\varepsilon_{i}^{j}$ es un encaje.
En esta situaci\'{o}n, cada $\varepsilon_{i}$ es un encaje.

Demos ahora algunas propiedades de sucesiones crecientes de espacios
topol\'{o}gicos (\cite{Glo2}, Lema 1.7):

\begin{proposition}
\label{P_TopologicalPropertiesAscendingSequenceTopologicalSpaces}Sea
$X_{1}\subset X_{2}\subset\cdots$ una sucesi\'{o}n creciente de espacios
topol\'{o}gicos tales que las inclusiones son continuas. Ponemos en
$X=\bigcup\limits_{n\in\mathbb{N}^{\ast}}X_{n}$ la topolog\'{\i}a final
correspondiente a la familia de las inclusiones $\varepsilon_{n}%
:X_{n}\hookrightarrow X$ (i.e. la DL-topolog\'{\i}a). Entonces, tenemos las
propiedades siguientes:
\begin{enumerate}
  \item Si cada $X_{n}$ es $T_{1}$, entonces $X$ es $T_{1}$.
  \item Si $O_{n}\subset X_{n}$ es abierto y $O_{1}\subset O_{2}\subset\cdots$, entonces $O=\bigcup\limits_{n\in\mathbb{N}^{\ast}}O_{n}$ es un abierto de $X$
y la DL-topolog\'{\i}a sobre $O=\underrightarrow{\lim}O_{n}$ coincide con la
topolog\'{\i}a inducida por $X$.
  \item Si cada $X_{n}$ es localmente compacto,
entonces $X$ es Haussdorf.
  \item Si cada $X_{n}$ es $T_{1}$ y $K\subset X$ es
compacto, entonces $K\subset X_{n}$ para cierto $n$.
\end{enumerate}
\end{proposition}

\subsection{L\'{\i}mites directos de espacios vectoriales topol\'{o}gicos de
dimension finita}

Sea $E$ un espacio vectorial real de dimensi\'{o}n numerable.

Sea $E_{1}\subset E_{2}\subset\cdots$ una sucesi\'{o}n creciente de
subespacios vectoriales de $E$ de dimensi\'{o}n finita tal que $E=\bigcup
\limits_{n\in\mathbb{N}^{\ast}}E_{n}$. Se obtiene una sucesi\'{o}n directa
estricta de subespacios vectoriales dotada de la topolog\'{\i}a\textit{ }final
por las inclusiones $E_{n}\hookrightarrow E$. \newline
Entonces $O\subset X$ es abierto si y s\'{o}lo si $\varepsilon_{n}^{-1}\left(  O\right)  $ es un
abierto de $X_{n}$ para cada $n\in\mathbb{N}$.

Puesto que los espacios son de dimensiones finitas esta topolog\'{\i}a
coincide con la m\'{a}s fina topolog\'{\i}a de espacio vectorial localmente
convexo (\cite{Glo1}, Ejemplo 3.5); el conjunto de todos los subconjuntos
equilibrados, absorbentes y convexos es una base para esta topolog\'{\i}a.
\newline Adem\'{a}s $E$ es un espacio vectorial conveniente (\cite{Glo1},
Lema 6.1).

Tenemos la relaci\'{o}n que sigue entre la diferenciabilidad $C^{\infty}$ y la
$c^{\infty}$ (\cite{Glo2}, Lema 1.9)

\begin{lemma}
Para una aplicaci\'{o}n $f:O\longrightarrow F$ donde $O=\bigcup\limits_{n\in
\mathbb{N}^{\ast}}O_{n}$ es una sucesi\'{o}n creciente de conjuntos abiertos
($O_{n}\subset E_{n}$) y $F$ un e.v.t.l.c. real y Mackey completo tenemos la
equivalencia siguiente: $f$ es $C^{\infty}$ si y s\'{o}lo si $f$ es
$c^{\infty}$.
\end{lemma}

\begin{example}
\label{Ex_Rinfinity}El espacio vectorial $\mathbb{R}^{\infty},$ tambi\'{e}n
denotado por $\mathbb{R}^{\left(  \mathbb{N}\right)  }$, de todas las
sucesiones finitas es el l\'{\i}mite directo estricto de $\left(
\mathbb{R}^{i},\varepsilon_{i}^{j}\right)  _{\left(  i,j\right)  \in
\mathbb{N}^{2},\ i\leq j}$ donde $\varepsilon_{i}^{j}:\left(  x_{1}%
,\dots,x_{i}\right)  \mapsto\left(  x_{1},\dots,x_{i},0,\dots,0\right)
.$\newline Es un espacio vectorial numerable y conveniente (\cite{KriMic},
47.1). Una base de $\mathbb{R}^{\infty}$ es $\left(  e_{i}\right)
_{i\in\mathbb{N}^{\ast}}$ donde $e_{i}=\left(  0,\dots,0,\underset
{i^{th}\ \text{term}}{1},0,\dots\right)  \in\varepsilon_{i}\left(
\mathbb{R}^{i}\right)  $.
\end{example}

\section{\label{_LimitesDirectosVariedades}L\'{\i}mites directos de
variedades}

\subsection{L\'{\i}mites directos de sucesiones crecientes de
variedades de dimensi\'{o}n finita}

Combinando los resultados obtenidos por Gl\"{o}ckner (\cite{Glo2}, Teorema 1 y
Proposici\'{o}n 3.6), obtenemos:

\begin{theorem}
\label{T_DirectLimitSequenceFiniteDimensionalParacompactManifolds}Sea
$\mathcal{M}=\left(  M_{i},\varepsilon_{i}^{j}\right)  _{i\in\mathbb{N}^{\ast
},\ j\in\mathbb{N}^{\ast},\ i\leq j}$ una sucesi\'{o}n directa de variedades
de clase $C^{\infty}$ paracompactas de dimensi\'{o}n finita donde
$\varepsilon_{i}^{j}:M_{i}\longrightarrow M_{j}$ son inmersiones inyectivas de
clase $C^{\infty}$. Sea $s=\underset{i\in\mathbb{N}^{\ast}}{\sup}\left(  \dim
M_{i}\right)  \in\mathbb{N}\cup\left\{  +\infty\right\}$. Hay una \'{u}nica
estructura de variedad $c^{\infty}$ modelada sobre $\mathbb{R}^{\infty}$ por
la cual $\varepsilon_{n}:M_{n}\longrightarrow M=\bigcup\limits_{i\in
\mathbb{N}^{\ast}}M_{i}$ es una aplicaci\'{o}n de
clase $c^{\infty}$ para cada $n\in\mathbb{N}^{\ast}$ y tal que $\left(
M,\varepsilon_{n}\right)  _{n\in\mathbb{N}^{\ast}}=\underrightarrow{\lim
}\mathcal{S}$ en la categor\'{\i}a de las variedades convenientes.
\end{theorem}

\begin{example}
\label{Ex_Sinfty}La esfera $\mathbb{S}^{\infty}$ (\cite{KriMic}, 47.2).-- El
espacio conveniente $\mathbb{R}^{\infty}$ es equipado con el producto escalar
d\'{e}bil dado por la suma finita $\left\langle x,y\right\rangle
=\sum\limits_{i}x_{i}y_{i}$ bilineal y acotada. El l\'{\i}mite inductivo de
$\mathbb{S}^{1}\subset\mathbb{S}^{2}\subset\cdots$ es el subconjunto cerrado
$\mathbb{S}^{\infty}=\left\{  x\in\mathbb{R}^{\infty}:\left\langle
x,x\right\rangle =1\right\}  $ de $\mathbb{R}^{\infty}$. Es una variedad
conveniente modelada sobre $\mathbb{R}^{\infty}$.
\end{example}

\subsection{Funciones sobre l\'{\i}mites directos de variedades}

Sea $\mathcal{M}=\left(  M_{i},\varepsilon_{i}^{j}\right)  _{i\in
\mathbb{N}^{\ast},\ j\in\mathbb{N}^{\ast},\ i\leq j}$ una sucesi\'{o}n directa
de variedades de clase $C^{\infty}$ donde $\varepsilon_{i}^{j}:M_{i}%
\longrightarrow M_{j}$ son inmersiones inyectivas $C^{\infty}$. Podemos
identificar $M_{i}$ con el subconjunto $\varepsilon_{i}^{j}\left(
M_{i}\right)  $ de $M_{j}$. El \'{a}lgebra de las funciones reales definidas sobre
$M_{i}$ se denotar\'{a} por $\mathcal{F}\left(  M_{i}\right)  $. Entonces, podemos definir la sucesi\'{o}n $\mathcal{F}=\left(  \mathcal{F}\left(
M_{i}\right)  ,\delta_{i}^{j}\right)  _{i\in\mathbb{N}^{\ast},\ j\in
\mathbb{N}^{\ast},\ i\leq j}$ con las aplicaciones
\[%
\begin{array}
[c]{cccc}%
\delta_{i}^{j}: & \mathcal{F}\left(  M_{j}\right)  & \longrightarrow &
\mathcal{F}\left(  M_{i}\right) \\
& f_{j} & \longmapsto & f_{i}%
\end{array}
\]
donde $\delta_{i}^{j}=\left(  \varepsilon_{i}^{j}\right)  ^{\ast}$, i.e.
$\delta_{i}^{j}\left(  f_{j}\right)  =f_{j}\circ\varepsilon_{i}^{j}$. Estas
aplicaciones satisfacen las condiciones $\delta_{i}^{j}\circ\delta_{j}%
^{k}=\delta_{i}^{k}$ para cada $i\leq j\leq k$. Entonces $\mathcal{F}$ es una
\textit{sucesi\'{o}n proyectiva} y se puede identificar el l\'{\i}mite
proyectivo $\underleftarrow{\lim}\mathcal{F}\left(  M_{i}\right)  $ con
$\bigcap\limits_{i=1}^{+\infty}\mathcal{F}\left(  M_{i}\right)  $.

Sea $f=\underleftarrow{\lim}f_{i}$ el l\'{\i}mite proyectivo de funciones
$f_{i}:M_{i}\longrightarrow\mathbb{R}$ de clase $C^{\infty}$. $f$ es $c^{\infty}$ (cf. \cite{Glo2}).\\
Para cada $x\in M$ se define la diferencial $df_{x}:T_{x}M\longrightarrow
\mathbb{R}$ de la manera siguiente: si $x\in M_{k}$, entonces $df\left(
x\right)  =df_{k}\left(  x_{k}\right)  $:$T_{x_{k}}M_{k}\longrightarrow
\mathbb{R}$.

\subsection{L\'{\i}mites directos de campos de vectores}

Sea $\mathcal{M}=\left(  M_{i},\varepsilon_{i}^{j}\right)  _{i\in
\mathbb{N}^{\ast},\ j\in\mathbb{N}^{\ast},\ i\leq j}$ una sucesi\'{o}n directa
de variedades paracompactas de clase $C^{\infty}$ donde $\varepsilon_{i}%
^{j}:M_{i}\longrightarrow M_{j}$ son inmersiones inyectivas $C^{\infty}$.
Entonces podemos equipar $\mathcal{T}=\left(  TM_{i},T\varepsilon_{i}%
^{j}\right)  _{i\in\mathbb{N}^{\ast},\ j\in\mathbb{N}^{\ast},\ i\leq j}$con
una estructura de fibrado vectorial conveniente (cf. \cite{SurCab}).

Sea $\left(  X_{i}\right)  _{i\in\mathbb{N}^{\ast}}$ una sucesi\'{o}n de
campos de vectores $X_{i}\in\mathfrak{X}\left(  M_{i}\right)  $ tal que:%
\[
T\varepsilon_{i}^{j}\circ X_{i}=X_{j}\circ\varepsilon_{i}^{j}%
\]

Podemos definir una secci\'{o}n $c^{\infty}$ de $\underrightarrow{\lim}TM_{i}$
como $\underrightarrow{\lim}X_{i}$.

\begin{example}
\label{Ex_VectorFieldOnRinfty}$X=\sum\limits_{i=1}^{+\infty}x_{i}%
\dfrac{\partial}{\partial x_{i}}$ es un campo de vectores de $\mathbb{R}%
^{\infty}$definido de la manera siguiente: sea $x\in\mathbb{R}^{\infty}$;
existe $n=n\left(  x\right)  \in\mathbb{N}^{\ast}$ tal que $x=\iota_{n}\left(
x_{1},\dots,x_{n}\right)  $ donde $\iota_{n}:\mathbb{R}^{n}\hookrightarrow
\mathbb{R}^{\infty}$ es la inyecci\'{o}n can\'{o}nica. Tenemos as\'{\i}
$X_{x}=\sum\limits_{i=1}^{n\left(  x\right)  }x_{i}\dfrac{\partial}{\partial
x_{i}}$.
\end{example}

\subsection{L\'{\i}mites directos de grupos de Lie}

Podemos encontrar en \cite{Glo2}, Teorema 4.3 (d), el resultado siguiente a
prop\'{o}sito del l\'{\i}mite directo de una sucesi\'{o}n de groupos de Lie de
dimension finita.

\begin{proposition}
\label{P_DirectLimitFiniteDimensionalLieGroup}Sea una sucesi\'{o}n de grupos
de Lie reales de dimension finita y de clase $C^{\infty}$ $\mathcal{G}=\left(
G_{i},\varepsilon_{i}^{j}\right)  _{i\in\mathbb{N}^{\ast},\ j\in
\mathbb{N}^{\ast},\ i\leq j}$ donde $\varepsilon_{i}^{j}:G_{i}\longrightarrow
G_{j}$ son $C^{\infty}$-homomorfismos. Entonces $G=\underrightarrow{\lim}%
G_{i}$ es un grupo de Lie $c^{\infty}$-regular.
\end{proposition}

\begin{example}
\label{Ex_GLRinfinity}$GL\left(  \infty,\mathbb{R}\right)  =\varinjlim
GL\left(  \mathbb{R}^{n}\right)  $ es un grupo de Lie conveniente.
\end{example}

\subsection{L\'{\i}mites directos de fibrados vectoriales}

Denotamos la inyecci\'{o}n can\'{o}nica $\mathbb{R}^{i}\hookrightarrow
\mathbb{R}^{j}$ por $\iota_{i}^{j}$.

Sea $\mathcal{M}=\left(  M_{i},\varepsilon_{i}^{j}\right)  _{i\in
\mathbb{N}^{\ast},\ j\in\mathbb{N}^{\ast},\ i\leq j}$ una sucesi\'{o}n directa
de variedades paracompactas de clase $C^{\infty}$ y de dimension finita ($\dim
M_{i}=d_{i}$) donde $\varepsilon_{i}^{j}:M_{i}\longrightarrow M_{j}$ son
inmersiones inyectivas de clase $C^{\infty}$. Podemos suponer que $M_{1}\subset
M_{2}\subset\cdots$. Aqu\'{\i} consideramos la situaci\'{o}n donde $\sup
d_{i}=+\infty$.

Par cada n\'{u}mero natural $i$, sea $\left(  E_{i},\pi_{i},M_{i}\right)  $ un
fibrado vectorial cuya fibra es isomorfa al espacio vectorial $\mathbb{E}_{i}$
de dimensi\'{o}n $r_{i}$ (que se puede identificar a $\mathbb{R}^{r_{i}}$)
donde $\underset{i\in\mathbb{N}^{\ast}}{\sup}r_{i}=+\infty$. Supongamos que
$\mathcal{E}=\left(  E_{i},\lambda_{i}^{j}\right)  _{i\in\mathbb{N}^{\ast
},\ j\in\mathbb{N}^{\ast},\ i\leq j}$ es una sucesi\'{o}n directa de
variedades donde $\lambda_{i}^{j}:E_{i}\longrightarrow E_{j}$ son inmersiones
inyectivas de clase $C^{\infty}$ y morfismos de fibrados vectoriales. Supongamos tambi\'{e}n que $E_{1}\subset E_{2}\subset\dots$.

A la sucesi\'{o}n $\mathcal{E}=\left(  \left(  E_{i},\pi_{i},M_{i}\right)
_{i},\lambda_{i}^{j}\right)  _{i\in\mathbb{N}^{\ast},\ j\in\mathbb{N}^{\ast
},\ i\leq j}$ la llamaremos \textit{l\'{\i}mite directo de fibrados
vectoriales} si, para cada $x_{i}\in M_{i}$, existe un sistema directo de
trivializaciones $\left(  U_{i},\Psi_{i}\right)  $ (con $U_{1}\subset
U_{2}\subset\cdots$ y $U_{i}$ abierto de $M_{i}$) de $\left(  E_{i},\pi
_{i},M_{i}\right)  $ donde $\Psi_{i}:\pi_{i}^{-1}\left(  U_{i}\right)
\longrightarrow U_{i}\times\mathbb{R}^{r_{i}}$ son difeomorfismos locales
tales que $x_{i}\in U_{i}$ y para cada par $\left(  i,j\right)  \in
\mathbb{N}^{2},i\leq j$, tenemos las condiciones de compatibilidad:%
\[
\left(  \varepsilon_{i}^{j}\times\iota_{r_{i}}^{r_{j}}\right)  \circ\Psi
_{i}=\Psi_{j}\circ\lambda_{i}^{j}.
\]

La proposici\'{o}n siguiente generaliza el resultado de \cite{SurCab} a
prop\'{o}sito de l\'{\i}mite directo de fibrados tangentes.

\begin{proposition}
\label{P_StructureOnDirectLimitVectorBundles} Sea $\mathcal{E}=\left(  \left(
E_{i},\pi_{i},M_{i}\right)  _{i},\lambda_{i}^{j}\right)  _{i\in\mathbb{N}%
^{\ast},\ j\in\mathbb{N}^{\ast},\ i\leq j}$ una sucesi\'{o}n directa de
fibrados vectoriales. Entonces $\left(  \underrightarrow{\lim}E_{i}%
,\underrightarrow{\lim}\pi_{i},\underrightarrow{\lim}M_{i}\right)  $ puede ser
equipado con una estructura de fibrado vectorial conveniente cuya base es
modelada sobre $\mathbb{R}^{\infty}$ y cuyo grupo estructural es el grupo de
Lie conveniente $GL\left(  \infty,\mathbb{R}\right)  =\varinjlim GL\left(
\mathbb{R}^{n}\right)  $.
\end{proposition}

\begin{proof}
Definamos en primer lugar una estructura de variedad sobre $E=\underrightarrow
{\lim}E_{i}$. \newline

Para cada $i\in\mathbb{N}^{\ast},$ $E_{i}$ es un espacio topol\'{o}gico
localmente compacto; entonces, $\underrightarrow{\lim}E_{i}$ es Hausdorff (cf.
\cite{Han}).

Sea $v\in E;$ definamos une carta en $v$ como sigue. Puesto que $v$ pertenece
a $\underrightarrow{\lim}E_{i}$, existe $n\in\mathbb{N}^{\ast}$ tal que
$v=\lambda_{n}\left(  v_{n}\right)  $ con $v_{n}\in E_{n}$ y $\pi_{n}\left(
v_{n}\right)  =x\in M_{n}$. Para $i\geq n$, la trivializaci\'{o}n local
$\Psi_{i}:\pi_{i}^{-1}\left(  U_{i}\right)  \longrightarrow U_{i}%
\times\mathbb{R}^{r_{i}}$ da origen, via una carta $\varphi_{i}:U_{i}%
\longrightarrow\mathbb{R}^{d_{i}}$, a una carta $\psi_{i}:\pi_{i}^{-1}\left(
U_{i}\right)  \longrightarrow\mathbb{R}^{d_{i}}\times\mathbb{R}^{r_{i}}$.
Puesto que $\lambda_{i}^{j}$ es un morfismo sobre $\varepsilon_{i}^{j}$,
tenemos $\pi_{j}\circ\lambda_{i}^{j}=\varepsilon_{i}^{j}\circ\pi_{i}$ y
podemos definir $\pi=\underrightarrow{\lim}\pi_{i}:\underrightarrow{\lim}%
E_{i}\longrightarrow\underrightarrow{\lim}M_{i}$ por $\pi\left(  v\right)
=\pi_{n}\left(  v_{n}\right)  =x\in M_{n}\subset\underrightarrow{\lim}M_{i}$.
Tenemos entonces $\underrightarrow{\lim}\left(  \pi_{i}\right)  ^{-1}\left(
U_{i}\right)  =\pi^{-1}\left(  \underrightarrow{\lim}U_{i}\right)  $.\newline
Puesto que el diagrama siguiente es conmutativo%
\[%
\begin{array}
[c]{ccc}%
\left(  \pi_{i}\right)  ^{-1}\left(  U_{i}\right)   & \underrightarrow
{\lambda_{i}^{j}} & \left(  \pi_{j}\right)  ^{-1}\left(  U_{j}\right)  \\
\Psi_{i}\downarrow &  & \downarrow\Psi_{j}\\
U_{i}\times\mathbb{R}^{r_{i}} & \underrightarrow{\left(  \varepsilon_{i}%
^{j}\times\iota_{r_{i}}^{r_{j}}\right)  } & U_{j}\times\mathbb{R}^{r_{j}}%
\end{array}
\]
podemos definir una trivializaci\'{o}n local del fibrado $\left(
\underrightarrow{\lim}E_{i},\underrightarrow{\lim}\pi_{i},\underrightarrow
{\lim}M_{i}\right)  $
\[%
\begin{array}
[c]{cccc}%
\Psi: & \pi^{-1}\left(  \underrightarrow{\lim}U_{i}\right)   & \longrightarrow
& \underrightarrow{\lim}U_{i}\times\mathbb{R}^{\infty}\\
& v=\underrightarrow{\lim}v_{i} & \mapsto & \left(  x=\underrightarrow{\lim
}x_{i},\widehat{v}=\underrightarrow{\lim}\widehat{v_{i}}\right)
\end{array}
\]
donde $\widehat{v_{i}}=\theta_{\iota}\left(  v_{i}\right)  $ con
$\theta_{\iota}=\operatorname*{pr}\nolimits_{2}\circ\Psi_{i}$. \newline$\Psi$
es un homeomorfismo como l\'{\i}mite directo de homeomorfismos.\newline
Utilizando las cartas $\left(  U_{i},\varphi_{i}\right)  _{i\geq n}$ se puede
definir una carta de $E$ en $v$:%
\[%
\begin{array}
[c]{cccc}%
\psi: & \pi^{-1}\left(  \underrightarrow{\lim}U_{i}\right)   & \longrightarrow
& \mathbb{R}^{\infty}\times\mathbb{R}^{\infty}\\
& v=\underrightarrow{\lim}v_{i} & \mapsto & \left(  \overline{x}%
=\underrightarrow{\lim}\overline{x_{i}},\widehat{v}=\underrightarrow{\lim
}\widehat{v_{i}}\right)
\end{array}
\]
donde $\overline{x_{i}}=\varphi_{i}\left(  x_{i}\right)  $.\newline Adem\'{a}s
$\theta_{x|\pi^{-1}\left(  y\right)  }=\underrightarrow{\lim}\left(
\theta_{i}\right)  _{x_{i}|\pi_{i}^{-1}\left(  y_{i}\right)  }:\pi^{-1}\left(
y\right)  \longrightarrow\left\{  y\right\}  \times\mathbb{R}^{\infty}$ es
lineal.\newline Probemos ahora que los cambios de cartas son $c^{\infty}%
$-difeomorfismos. \newline Consideramos dos cartas $\left(  \pi^{-1}\left(
U^{\alpha}\right)  ,\psi^{\alpha}\right)  $ y $\left(  \pi^{-1}\left(
U^{\beta}\right)  ,\psi^{\beta}\right)  $ en $v\in E$. Tenemos entonces:%
\[
\tau^{\beta\alpha}=\psi^{\beta}\circ\left(  \psi^{\alpha}\right)  ^{-1}%
:\psi^{\alpha}\left(  \pi^{-1}\left(  U^{\alpha}\cap U^{\beta}\right)
\right)  \longrightarrow\psi^{\beta}\left(  \pi^{-1}\left(  U^{\alpha}\cap
U^{\beta}\right)  \right)
\]
donde $\tau^{\alpha\beta}\circ\left(  \iota_{d_{i}},\iota_{r_{i}}\right)
=\left(  \iota_{d_{i}},\iota_{r_{i}}\right)  \circ\tau_{i}^{\alpha\beta}$
($\iota_{k}:\mathbb{R}^{k}$ $\hookrightarrow\mathbb{R}^{\infty}$ es la
inyecci\'{o}n can\'{o}nica). Se sigue que $\tau^{\beta\alpha}=\left(
\varphi^{\beta}\circ\left(  \varphi^{\alpha}\right)  ^{-1},\theta^{\beta}%
\circ\left(  \theta^{\alpha}\right)  ^{-1}\right)  \in\operatorname{Diff}%
^{\infty}\left(  \mathbb{R}^{\infty}\right)  \times GL\left(  \mathbb{R}%
,\infty\right)  $ donde $Gl\left(  \mathbb{R},\infty\right)  $ es un grupo de
Lie conveniente (\cite{KriMic}, 47.7).
\newline
Probemos ahora la $c^{\infty}%
$-diferenciabilidad de la proyecci\'{o}n $\pi:E\longrightarrow M$. Sea
$c:\left]  -\varepsilon,\varepsilon\right[  \longrightarrow\pi^{-1}\left(
U\right)  \in E$ una curva tale que $c\left(  0\right)  =v$ donde $\pi\left(
v\right)  =x\in U\subset M$. Puesto que $c\left(  \left]  -\varepsilon
,\varepsilon\right[  \right)  $ es relativamente compacto, tenemos $c\left(
\left]  -\varepsilon,\varepsilon\right[  \right)  \subset\left(  \pi
_{n}\right)  ^{-1}\left(  U_{n}\right)  $
(\ref{P_TopologicalPropertiesAscendingSequenceTopologicalSpaces}, 4.).
Entonces $\pi\circ c=\pi_{n}\circ c_{n}\in C^{\infty}\left(  \left]
-\varepsilon,\varepsilon\right[  ,\mathbb{R}^{d_{n}}\right).$ Deducimos de
esto que $\pi$ es $c^{\infty}.$
\end{proof}

\section{\label{_AlgebroidesLie}Algebroides de Lie}

\bigskip Recordamos definiciones y propiedades utilizadas en \cite{CabPel}
(ver tambi\'{e}n \cite{Igl} y \cite{Mar2}).

\subsection{Definiciones. Expresi\'{o}n local. Ejemplos}

Sea $\tau:E\rightarrow M$ un fibrado vectorial sobre une variedad de
dimensi\'{o}n finita cuya fibra es un espacio vectorial $\mathbb{E}$ de
dimensi\'{o}n finita.

Un morfismo de fibrados vectoriales $\rho:E\rightarrow TM$ se llama un
\textit{ancla}. Este morfismo da origen a un $\mathcal{F}-$modulo
\underline{$\rho$}$:\Gamma\left(  E\right)  \rightarrow\Gamma\left(
TM\right)  $=$\mathfrak{X}(M)$ definido para cada $x$ $\in$ $M$ y cada
secci\'{o}n $s$ de $E$ por: $\left(  \underline{\rho}\left(  s\right)
\right)  \left(  x\right)  =\rho\left(  s\left(  x\right)  \right)  $ y
denotado tambi\'{e}n por $\rho$.

Llamaremos a $\left(  E,\tau,M,\rho\right)  $ un \textit{fibrado anclado}.

\begin{definition}
Un casi-corchete sobre un fibrado \textit{anclado} $\left(  E,\tau
,M,\rho\right)  $ es un corchete $\left[  .,.\right]  _{\rho}$ que satisface
una regla de Leibniz:%
\[
\forall f\in\mathcal{F},\forall s_{1},s_{2}\in\Gamma\left(  E\right)
,\quad\left[  s_{1},fs_{2}\right]  _{\rho}=f\left[  s_{1},s_{2}\right]
_{\rho}+\left(  \rho\left(  s_{1}\right)  \right)  \left(  f\right)  \ s_{2}%
\]
Un fibrado \textit{anclado} $\left(  E,\tau,M,\rho\right)  $ equipado con un
casi-corchete es llamado casi-algebroide de Lie.
\end{definition}

\begin{definition}
Un corchete de Lie sobre un fibrado \textit{anclado} $\left(  E,\tau
,M,\rho\right)  $ es un casi-corchete que verifica la identidad de Jacobi:%
\[
\forall s_{1},s_{2,}s_{3}\in\Gamma\left(  E\right)  ,\; \left[  s_{1},\left[  s_{2},s_{3}\right]  _{\rho
}\right]  _{\rho}+\left[  s_{2},\left[  s_{3},s_{1}\right]  _{\rho}\right]
_{\rho}+\left[  s_{3},\left[  s_{1},s_{2}\right]  _{\rho}\right]  _{\rho}=0
\]
Un fibrado \textit{anclado} $\left(  E,\pi,M,\rho\right)$ equipado con un
corchete de Lie se llama algebroide de Lie.
\end{definition}

Cuando $\left(  E,\tau,M,\rho\right)  $ es un algebroide de Lie el corchete
$\left[  .,.\right]  _{\rho}$ es un morfismo de \'{a}lgebras de Lie.

\begin{example}
\bigskip Cualquier \'{a}lgebra de Lie de dimensi\'{o}n finita es un algebroide
de Lie sobre un punto.
\end{example}

\begin{example}
El fibrado tangente $\pi_{M}:TM\longrightarrow M$ a una variedad diferencial
es un algebroide de Lie donde el ancla es $\operatorname{Id}_{TM}$ y el
corchete de secciones es el corchete de Lie de los campos de vectores.
\end{example}

\begin{example}
\label{Ex_LieAlgebroid_NijenhuisTensor}{$E=TM$ y $\rho=N$ es un tensor de
Nijenhuis, i.e. un tensor que verifica la condici\'{o}n%
\[
\left[  NX,NY\right]  =N\left(  \left[  NX,Y\right]  +\left[  X,NY\right]
-N\left(  \left[  X,Y\right]  \right)  \right)
\]
$\left(  TM,\pi,M,N\right)  $ es un algebroide de Lie (cf. \cite{Cab}) donde
el ancla es }$N$ y el corchete de secciones es el corchete{ $\left[  .,.\right]
_{N}$ definido por: \newline%
\[
\left[  X,Y\right]  _{N}=\left[  NX,Y\right]  +\left[  X,NY\right]  -N\left(
\left[  X,Y\right]  \right)
\]
}
\end{example}

\begin{example}
\label{Ex_LieAlgebroid_CtangentBundlePoissonTensor}Sea $\left(  M,\Lambda
\right)  $ una variedad de Poisson. El fibrado cotangente {$T^{\ast}M$ a una
variedad diferencial est\'{a} dotado de una estructura de algebroide de Lie
donde el ancla es el morfismo }$\Lambda^{\sharp}${ y el }corchete en las
1-formas{ (cf. \cite{MagMor}) viene dado por:
\[
\left[  \alpha,\beta\right]  _{\Lambda}=L_{\Lambda^{\sharp}\beta}\left(
\alpha\right)  -L_{\Lambda^{\sharp}\alpha}\left(  \beta\right)  +d\left\langle
\beta,\Lambda^{\sharp}\alpha\right\rangle
\]
}
\newline\newline para $\alpha,\beta\in\Omega^{1}\left(  M\right)  $.
\end{example}

\bigskip

Si fijamos un sistema de coordenadas locales $\left(  x^{i}\right)  _{1\leq
i\leq n}$ en la base $M$ y una base local $\left(  e_{\alpha}\right)
_{1\leq\alpha\leq m}$ de secciones de $\tau$, obtenemos un sistema de
coordonedas locales $\left(  x^{i},y^{\alpha}\right)  $ de $E$.

Tenemos una expresi\'{o}n local del ancla y del corchete:%
\[
\rho\left(  e_{\alpha}\right)  =\rho_{\alpha}^{i}\dfrac{\partial}{\partial
x^{i}}\text{ y }\left[  e_{\alpha},e_{\beta}\right]  _{\rho}=C_{\alpha\beta
}^{\gamma}e_{\gamma}%
\]

donde las funciones $\rho_{\alpha}^{i}$ y $C_{\alpha\beta}^{\gamma}$ verifican
relaciones debidas a la condici\'{o}n de compatibilidad y a la identidad de Jacobi.

\subsection{C\'{a}lculo en algebroides de Lie}

Dado un algebroide de Lie $\left(  E,\tau,M,\rho\right)  $ se definen la
derivada de Lie y la diferencial exterior.

\begin{description}
\item Derivada de Lie

Sea $s$ una secci\'{o}n de $E.$ Se define la \textit{derivada de Lie}
$L_{s}^{\rho}$

-- de una funci\'{o}n $f\in\Omega^{0}\left(  M,E\right)  =\mathcal{F}$ de
clase $C^{\infty}$ por:
\[
L_{s}^{\rho}\left(  f\right)  =L_{\rho\circ s}\left(  f\right)  =i_{\rho\circ
s}\left(  df\right)  \newline%
\]

-- de una $q$--forma $\omega\in\Omega^{q}\left(  M,E\right)  $ (donde $q>0$)
por%
\begin{align}
\left(  L_{s}^{\rho}\omega\right)  \left(  s_{1},\dots,s_{q}\right)   &
=L_{s}^{\rho}\left(  \omega\left(  s_{1},\dots,s_{q}\right)  \right)
\nonumber\\
&  -{\sum\limits_{i=1}^{q}}\omega\left(  s_{1},\dots,s_{i-1},\left[
s,s_{i}\right]  _{\rho},s_{i+1},\dots,s_{q}\right). \nonumber
\end{align}

\item Diferencial exterior

Se define la \textit{diferencial exterior} $d_{\rho}$

-- de una funci\'{o}n $f\in\mathcal{F}$ por%
\[
d_{\rho}f=t_{\rho}\circ df
\]

-- de una $q$--forma $\omega\in\Omega^{q}\left(  M,E\right)  $ (donde $q>0$)
por%
\begin{align*}
\left(  d_{\rho}\omega\right)  \left(  s_{0},\dots,s_{q}\right)   &
={\sum\limits_{i=0}^{q}}\left(  -1\right)  ^{i}L_{s_{i}}^{\rho}\left(
\omega\left(  s_{0},\dots,\widehat{s_{i}},\dots,s_{q}\right)  \right) \\
&  +{\sum\limits_{0\leq i<j\leq q}^{q}}\left(  -1\right)  ^{i+j}\left(
\omega\left(  \left[  s_{i},s_{j}\right]  _{\rho},s_{0},\dots,\widehat{s_{i}%
},\dots,\widehat{s_{j}},\dots,s_{q}\right)  \right).
\end{align*}

\end{description}

\bigskip

Si $\left(  x^{i}\right)  _{1\leq i\leq n}$ son coordenadas locales sobre $M$
y si $\left\{  e_{\alpha}\right\}  _{1\leq\alpha\leq m}$ es una base local de secciones de $\tau$ tenemos las expresiones locales siguientes:%
\[
dx^{i}=\rho_{\alpha}^{i}e^{\alpha}\text{ y }de^{\gamma}=-\dfrac{1}{2}%
C_{\alpha\beta}^{\gamma}e^{\alpha}\wedge e^{\beta}%
\]

donde $\left(  e^{\alpha}\right)  $ es la base dual de $\left(  e_{\alpha
}\right)  $.

Tenemos
\[
d_{\rho}\circ d_{\rho}=0.
\]

La cohomolog\'{\i}a asociada con $d_{\rho}$ se denomina la cohomolog\'{\i}a
del algebroide de Lie $\left(  E,\tau,M,\rho\right)  .$

\subsection{Morfismo de algebroides de Lie}

\begin{definition}
Un morfismo de fibrados vectoriales $\psi:E\rightarrow E^{\prime}$ sobre
$f:M\rightarrow M^{\prime}$ es un morfismo de algebroides de Lie $\left(
E,\tau,M,\rho\right)  $ y $\left(  E^{\prime},\tau^{\prime},M^{\prime}%
,\rho^{\prime}\right)  $ si $\psi^{\ast}:\Omega^{q}\left(  M,E^{\prime
}\right)  \rightarrow\Omega^{q}\left(  M,E\right)  $ definida por:
\[
\left(  \psi^{\ast}\alpha^{\prime}\right)  _{x}\left(  s_{1},\dots
,s_{q}\right)  =\alpha_{f\left(  x\right)  }^{\prime}\left(  \psi\circ
s_{1},\dots,\psi\circ s_{q}\right)
\]
verifica la relaci\'{o}n:%
\[
d_{\rho}\circ\psi^{\ast}=\psi^{\ast}\circ d_{\rho^{\prime}}%
\]

\end{definition}

\subsection{Prolongaciones de algebroides de Lie}

Vamos a recordar la definici\'{o}n de la estructura de algebroide de Lie sobre
la prolongaci\'{o}n de un algebroide de Lie mediante una fibraci\'{o}n
(v\'{e}ase \cite{LMM}, \cite{Mar2} y \cite{Nie}).

Sea $\left(  E,\tau,M,\rho\right)  $ un algebroide de Lie de rango $m$ sobre
una variedad $M$ de dimensi\'{o}n $n$ y sea $\nu:P\longrightarrow M$ una fibraci\'{o}n de rango
$q$, esto es, una submersi\'{o}n sobreyectiva.

Para cada punto $p\in P$, consideramos el conjunto
\[
\mathcal{T}_{p}^{E}P=\left\{  \left(  b,v\right)  \in E_{x}\times T_{p}%
P:\rho\left(  b\right)  =T_{p}\nu\left(  v\right)  \right\}
\]

donde $T\nu:TP\longrightarrow TM$ es la applicaci\'{o}n tangente a $\nu$ y
$\nu\left(  p\right)  =x\in M$.

El conjunto $\mathcal{T}^{E}P=\bigcup\limits_{p\in P}\mathcal{T}_{p}^{E}P$
tiene una estructura natural de fibrado vectorial cuya proyecci\'{o}n es
$\tau_{P}^{E}:\left(  b,v\right)  \mapsto\nu\left(  v\right)  $.

Denotaremos de manera redundante $\left(  b,v\right)  $ por $\left(
p,b,v\right)  $.

El fibrado vectorial $\tau_{P}^{E}:\mathcal{T}^{E}P\rightarrow P$ admite una
estructura de algebroide de Lie denominada \textit{prolongaci\'{o}n del
algebroide de Lie }$E$\textit{ mediante la fibraci\'{o}n }$\nu$ o
\textit{fibrado }$E$\textit{-tangente a }$P$.

El ancla $\rho_{P}$:$\mathcal{T}^{E}P\rightarrow TP$ es la proyecci\'{o}n
sobre el tercer factor, i.e. $\rho_{P}\left(  p,b,v\right)  =v$.

Con el fin de definir un corchete de secciones de $\tau_{P}^{E}$ vamos a
considerar secciones particulares.

Una secci\'{o}n $Z\in\Gamma\left(  \tau_{P}^{E}\right)  $ se dice que es
\textit{proyectable} si existe una secci\'{o}n $\sigma$ de $\tau
:E\longrightarrow M$ y un campo de vectores $U$ sobre $P$ proyectable mediante
al campo de vectores $\rho\left(  \sigma\right)  $ y tales que $Z\left(
p\right)  =\left(  p,\sigma\left(  \nu\left(  p\right)  \right)  ,U\left(
p\right)  \right)  $ para todo $p\in P$.

El corchete de dos secciones $Z_{1}$ y $Z_{2}$ dadas por $Z_{i}\left(
p\right)  =\left(  p,\sigma_{i}\left(  \nu\left(  p\right)  \right)
,U_{i}\left(  p\right)  \right)  $, $i=1,2$, es:%
\[
\left[  Z_{1},Z_{2}\right]  _{\rho_{P}}\left(  p\right)  =\left(  p,\left[
\sigma_{1},\sigma_{2}\right]  _{\rho}\left(  \nu\left(  p\right)  \right)
,\left[  U_{1},U_{2}\right]  \left(  p\right)  \right)
\]

para cada $p\in P$.

Es f\'{a}cil probar que se puede elegir una base local de secciones
proyectables del espacio $\Gamma\left(  \tau_{P}^{E}\right)  $.

Si $\left(  x^{i},u^{A}\right)  _{1\leq i\leq n,1\leq A\leq q}$ son
coordenadas locales sobre $P$ y si $\left\{  e_{\alpha}\right\}  _{1\leq
\alpha\leq m}$ es una base local de secciones de $\tau:E\longrightarrow M$
podemos definir una base local $\left\{  \mathcal{X}_{\alpha},\mathcal{V}%
_{A}\right\}  _{1\leq\alpha\leq m,1\leq A\leq q}$de secciones de $\tau_{P}%
^{E}$ dadas por:%
\[
\mathcal{X}_{\alpha}\left(  p\right)  =\left(  p,e_{\alpha}\left(  \nu\left(
p\right)  \right)  ,\rho_{\alpha}^{i}\left(  \dfrac{\partial}{\partial x^{i}%
}\right)  _{p}\right)  \qquad\text{y\qquad}\mathcal{V}_{A}\left(  p\right)
=\left(  p,0,\left(  \dfrac{\partial}{\partial u^{A}}\right)  _{p}\right)
\]

Si $z=\left(  p,b,v\right)  $ partenece a $\mathcal{T}^{E}P$ donde
$b=z^{\alpha}e_{\alpha}$, entonces $v$ es de la forma
\[
v=\rho_{\alpha}^{i}z^{\alpha}\dfrac{\partial}{\partial x^{i}}+v^{A}%
\dfrac{\partial}{\partial u^{A}}%
\]

y podemos escribir:%
\[
z=z^{a}\mathcal{X}_{\alpha}\left(  p\right)  +v^{A}\mathcal{V}_{A}\left(
p\right)  \text{.}%
\]

Para $Z\in\Gamma\left(  \tau_{P}^{E}\right)  $ dada localmente por
$Z=Z^{\alpha}\mathcal{X}_{\alpha}+V^{A}\mathcal{V}_{A}$ se tiene que%
\[
\rho_{P}\left(  Z\right)  =\rho_{\alpha}^{i}Z^{\alpha}\dfrac{\partial
}{\partial x^{i}}+V^{A}\dfrac{\partial}{\partial u^{A}}%
\]

\subsection{Prolongaci\'{o}n de morfismos de algebroides de Lie}

Sean $\nu:P\longrightarrow M$ y $\nu^{\prime}:P^{\prime}\longrightarrow
M^{\prime}$ dos fibraciones. Sea $\Psi:P\longrightarrow P^{\prime}$ una
aplicaci\'{o}n fibrada sobre $\varphi:M\longrightarrow M^{\prime}$.
Consideramos dos algebroides de Lie $\tau:E\longrightarrow M$ y $\tau^{\prime
}:E^{\prime}\longrightarrow M^{\prime}$ y una aplicaci\'{o}n $\Phi
:E\longrightarrow E^{\prime}$ fibrada sobre $\varphi$. Si $\Phi$ es admisible
podemos definir una aplicaci\'{o}n admisible $T^{\Phi}\Psi:\mathcal{T}%
^{E}P\longrightarrow\mathcal{T}^{E^{\prime}}P^{\prime}$ por%
\[
T^{\Phi}\Psi\left(  p,b,v\right)  =\left(  \Psi\left(  p\right)  ,\Phi\left(
b\right)  ,T\Psi\left(  v\right)  \right)  \text{.}%
\]

$\bigskip$Recordamos el resultado siguiente (v\'{e}ase \cite{Mar1}):

\begin{proposition}
\label{P_MorphismExtensionLieAlgebroids}$T^{\Phi}\Psi$ es un morfismo de
algebroides de Lie si y s\'{o}lo si $\Phi$ es un morfismo de algebroides de
Lie .
\end{proposition}

\section{\label{_LimitesDirectosProlongacionesAlgebroidesLie}L\'{\i}mites
directos de prolongaciones de algebroides de Lie}

\begin{definition}
Se llama a $\left(  E_{i},\tau_{i},M_{i},\rho_{i}\right)
_{i\in\mathbb{N}^{\ast}}$ sucesi\'{o}n directa de algebroides de Lie si
\newline1. $\left(  \left(  E_{i},\lambda_{i}^{j}\right)  \right)
_{i\in\mathbb{N}^{\ast},\ j\in\mathbb{N}^{\ast},\ i\leq j}$ es una
sucesi\'{o}n directa de fibrados vectoriales de dimensiones finitas ($\tau
_{i}:E_{i}\rightarrow M_{i}$) sobre la sucesi\'{o}n directa de variedades de
dimensiones finitas $\left(  \left(  M_{i},\varepsilon_{i}^{j}\right)
\right)  _{i\in\mathbb{N}^{\ast},\ j\in\mathbb{N}^{\ast},\ i\leq j}$\newline2.
Para cada $i,j\in\mathbb{N}^{\ast}$ tales que $i\leq j$, tenemos%
\[
\rho_{j}\circ\lambda_{i}^{j}=T\varepsilon_{i}^{j}\circ\rho_{i}%
\]
\newline3. $\lambda_{i}^{j}:E_{i}\rightarrow E_{j}$ es un morfismo de los
algebroides de Lie $\left(  E_{i},\tau_{i},M_{i},\rho_{i}\right)  $ y
$\left(  E_{j},\tau_{j},M_{j},\rho_{j}\right)  .$
\end{definition}

Tenemos el resultado suiguiente:

\begin{theorem}
\label{T_ConvenienLieAlgebroid}Sea $\left(  E_{i},\tau_{i},M_{i},\rho
_{i}\right)  _{i\in\mathbb{N}^{\ast}}$ una sucesi\'{o}n directa de algebroides de
Lie. \newline Entonces $\left(  \underrightarrow{\lim}E_{i},\underrightarrow
{\lim}\tau_{i},\underrightarrow{\lim}M_{i},\underrightarrow{\lim}\rho
_{i}\right)  $ es un algebroide de Lie conveniente.
\end{theorem}

\begin{proof}

1. $\left(  \underrightarrow{\lim}E_{i},\underrightarrow{\lim}\tau
_{i},\underrightarrow{\lim}M_{i}\right)  $ es un fibrado vectorial conveniente
sobre la variedad conveniente $\underrightarrow{\lim}M_{i}$ modelada sobre
$\mathbb{R}^{\infty}$ (cf. Proposici\'{o}n
\ref{P_StructureOnDirectLimitVectorBundles}).

2. Sean $\left(  s_{i}^{1}\right)  _{i\in\mathbb{N}^{\ast}}$ y $\left(
s_{i}^{2}\right)  _{i\in\mathbb{N}^{\ast}}$ directas sucesiones de secciones de fibrados vectoriales $\pi_{i}:E_{i}\rightarrow M_{i}$. Entonces se cumplen las condiciones
\begin{equation}
\left\{
\begin{array}
[c]{c}%
\lambda_{i}^{j}\circ s_{i}^{1}=s_{j}^{1}\circ\varepsilon_{i}^{j}\\
\lambda_{i}^{j}\circ s_{i}^{2}=s_{j}^{2}\circ\varepsilon_{i}^{j}%
\end{array}
\right.  \label{_Comp_Sect}%
\end{equation}

Tenemos que probar la compatibilidad de los corchetes:
\begin{equation}
\lambda_{i}^{j}\circ\left[  s_{i}^{1},s_{i}^{2}\right]  _{E_{i}}=\left[
s_{j}^{1},s_{j}^{2}\right]  _{E_{j}}\circ\varepsilon_{i}^{j}
\label{_Comp_SectBrackets}%
\end{equation}

y de las propiedas de Leibniz:
\begin{equation}
\lambda_{i}^{j}\circ\left[  s_{i}^{1},g_{i}\times s_{i}^{2}\right]  _{E_{i}%
}=\left[  s_{j}^{1},g_{j}\times s_{j}^{2}\right]  _{E_{j}}\circ\varepsilon
_{i}^{j} \label{_Comp_SectLeibniz}%
\end{equation}

a) Con el fin de probar la compatibilidad de los corchetes utilizamos los morfismos
$\lambda_{i}^{j}:E_{i}\longrightarrow E_{j}$ de algebroides de Lie sobre
$\varepsilon_{i}^{j}:M_{i}\longrightarrow M_{j}$ que satisfacen:
\begin{equation}
d_{\rho_{i}}\circ\left(  \lambda_{i}^{j}\right)  ^{\ast}=\left(  \lambda
_{i}^{j}\right)  ^{\ast}\circ d_{\rho_{j}} \label{_MorphismLieAlgebroids}%
\end{equation}

y por tanto, aplicados a $\alpha_{j}\in\Omega^{1}\left(  M_{j},E_{j}\right):$

\[
\left(  d_{\rho_{i}}\circ\left(  \lambda_{i}^{j}\right)  ^{\ast
}\left(  \alpha_{j}\right)  \right)  \left(  s_{i}^{1},s_{i}^{2}\right)
=\left(  \left(  \lambda_{i}^{j}\right)  ^{\ast}\circ d_{\rho_{j}}\left(
\alpha_{j}\right)  \right)  \left(  s_{i}^{1},s_{i}^{2}\right).
\]
\newline\newline El primer miembro es igual a
\begin{align*}
&  \left(  d_{\rho_{i}}\circ\left(  \lambda_{i}^{j}\right)  ^{\ast}\left(
\alpha_{j}\right)  \right)  \left(  s_{i}^{1},s_{i}^{2}\right) \\
&  =L_{\rho_{i}\circ s_{i}^{1}}\left(  \left(  \left(  \lambda_{i}^{j}\right)
^{\ast}\left(  \alpha_{j}\right)  \right)  \left(  s_{i}^{2}\right)  \right)
-L_{\rho_{i}\circ s_{i}^{2}}\left(  \left(  \left(  \lambda_{i}^{j}\right)
^{\ast}\left(  \alpha_{j}\right)  \right)  \left(  s_{i}^{1}\right)  \right)
-\left(  \left(  \lambda_{i}^{j}\right)  ^{\ast}\left(  \alpha_{j}\right)
\right)  \left[  s_{i}^{1},s_{i}^{2}\right]  _{E_{i}}\\
&  =X_{j}^{1}\left(  \alpha_{j}\left(  \lambda_{i}^{j}\circ s_{i}^{2}\right)
\right)  -X_{j}^{2}\left(  \alpha_{j}\left(  \lambda_{i}^{j}\circ s_{j}%
^{1}\right)  \right)  -\alpha_{j}\left(  \lambda_{i}^{j}\circ\left[  s_{i}%
^{1},s_{i}^{2}\right]  _{E_{i}}\right)
\end{align*}

donde $X_{j}^{a}=\rho_{j}\circ s_{j}^{a}$ con $a=1,2$ cumplen la relaci\'{o}n:
$X_{j}^{a}\left(  f_{j}\right)  =X_{i}^{a}\left(  f_{i}\right)$ donde
$f_{j}=\alpha_{j}\circ s_{j}$.

El segundo miembro es igual a
\begin{align*}
&  \left(  \left(  \lambda_{i}^{j}\right)  ^{\ast}\left(  d_{\rho_{j}}\left(
\alpha_{j}\right)  \right)  \right)  \left(  s_{i}^{1},s_{i}^{2}\right) \\
&  =d_{\rho_{j}}\left(  \alpha_{j}\right)  \left(  \lambda_{i}^{j}\circ
s_{i}^{1},\lambda_{i}^{j}\circ s_{i}^{2}\right) \\
&  =L_{\rho_{j}\circ\lambda_{i}^{j}\circ s_{i}^{1}}\left(  \alpha_{j}\left(
\lambda_{i}^{j}\circ s_{i}^{2}\right)  \right)  -L_{\rho_{j}\circ\lambda
_{i}^{j}\circ s_{i}^{2}}\left(  \alpha_{j}\left(  \lambda_{i}^{j}\circ
s_{i}^{1}\right)  \right)  -\alpha_{j}\left[  \lambda_{i}^{j}\circ s_{i}%
^{1},\lambda_{i}^{j}\circ s_{i}^{2}\right]  _{E_{j}}\\
&  =L_{\rho_{j}\circ s_{j}^{1}}\left(  \alpha_{j}\left(  \lambda_{i}^{j}\circ
s_{i}^{2}\right)  \right)  -L_{\rho_{j}\circ s_{j}^{2}}\left(  \alpha
_{j}\left(  \lambda_{i}^{j}\circ s_{i}^{1}\right)  \right)  -\alpha_{j}\left[
\lambda_{i}^{j}\circ s_{i}^{1},\lambda_{i}^{j}\circ s_{i}^{2}\right]  _{E_{j}%
}\\
&  =X_{j}^{1}\left(  \alpha_{j}\left(  \lambda_{i}^{j}\circ s_{i}^{2}\right)
\right)  -X_{j}^{2}\left(  \alpha_{j}\left(  \lambda_{i}^{j}\circ s_{j}%
^{1}\right)  \right)  -\alpha_{j}\left[  \lambda_{i}^{j}\circ s_{i}%
^{1},\lambda_{i}^{j}\circ s_{i}^{2}\right]  _{E_{j}}%
\end{align*}

En particular, para cada $\alpha_{j}\in\Omega^{1}\left(  M_{j},E_{j}\right)$,
\newline$\alpha_{j}\left(  \lambda_{i}^{j}\left(  \left[  s_{i}^{1},s_{i}%
^{2}\right]  _{E_{i}}\right)  \right)  =\alpha_{j}\left[  \lambda_{i}^{j}\circ
s_{i}^{1},\lambda_{i}^{j}\circ s_{i}^{2}\right]  _{E_{j}}$ y por lo tanto:\\
$\lambda_{i}^{j}\circ\left[  s_{i}^{1},s_{i}^{2}\right]  _{E_{i}}=\left[
\lambda_{i}^{j}\circ s_{i}^{1},\lambda_{i}^{j}\circ s_{i}^{2}\right]  _{E_{j}%
}$.\newline Utilizando $\lambda_{i}^{j}\circ s_{i}^{a}=s_{j}^{a}\circ
\varepsilon_{i}^{j}$, tenemos: $\lambda_{i}^{j}\circ\left[  s_{i}^{1}%
,s_{i}^{2}\right]  _{E_{i}}=\left[  s_{j}^{1},s_{j}^{2}\right]  _{E_{j}}%
\circ\varepsilon_{i}^{j}$.

b) Con el fin de probar  (\ref{_Comp_SectLeibniz}) vamos a establecer
\[
\lambda_{i}^{j}\circ\left(  g_{i}\times\left[  s_{i}^{1},s_{i}^{2}\right]
_{E_{i}}+\left(  \rho_{i}\left(  s_{i}^{1}\right)  \right)  \left(
g_{i}\right)  \times s_{i}^{2}\right)  =\left(  g_{j}\times\left[  s_{j}%
^{1},s_{j}^{2}\right]  _{E_{j}}+\left(  \rho_{j}\left(  s_{j}^{1}\right)
\right)  \left(  g_{j}\right)  \times s_{j}^{2}\right)  \circ\varepsilon
_{i}^{j}%
\]
porque, para $k \in \{i,j\}$, tenemos:
\[
\left[  s_{k}^{1},g_k \times s_{k}^{2}\right]
_{E_{k}}=
g_{k}\times\left[  s_{k}^{1},s_{k}^{2}\right]
_{E_{k}}+\left(  \rho_{k}\left(  s_{k}^{1}\right)  \right)  \left(
g_{k}\right)  \times s_{k}^{2}
\]

Podemos escribir:
\begin{align*}
&  \lambda_{i}^{j}\circ\left(  g_{i}\times\left[  s_{i}^{1},s_{i}^{2}\right]
_{E_{i}}+\left(  \rho_{i}\left(  s_{i}^{1}\right)  \right)  \left(
g_{i}\right)  \times s_{i}^{2}\right) \\
&  =\lambda_{i}^{j}\circ\left(  g_{i}\times\left[  s_{i}^{1},s_{i}^{2}\right]
_{E_{i}}\right)  +\lambda_{i}^{j}\circ\left(  \left(  \rho_{i}\left(
s_{i}^{1}\right)  \right)  \left(  g_{i}\right)  \times s_{i}^{2}\right) \\
&  =g_{i}\times\left(  \lambda_{i}^{j}\circ\left[  s_{i}^{1},s_{i}^{2}\right]
_{E_{i}}\right)  +\lambda_{i}^{j}\left(  X_{i}^{1}\left(  g_{i}\right)
\right)  \times\lambda_{i}^{j}\circ s_{i}^{2}\quad\text{(}\lambda_{i}%
^{j}\text{ es un morfismo)}\\
&  =g_{i}\times\left(  \left[  s_{j}^{1},s_{j}^{2}\right]  _{E_{j}}%
\circ\varepsilon_{i}^{j}\right)  +X_{j}^{1}\left(  g_{j}\right)
\circ\varepsilon_{i}^{j}\times s_{j}^{2}\circ\varepsilon_{i}^{j}\quad\text{cf.
(\ref{_Comp_SectBrackets})}\\
&  =\left(  g_{j}\circ\varepsilon_{i}^{j}\right)  \times\left(  \left[
s_{j}^{1},s_{j}^{2}\right]  _{E_{j}}\circ\varepsilon_{i}^{j}\right)  +\left(
X_{j}^{1}\left(  g_{j}\right)  \times s_{j}^{2}\right)  \circ\varepsilon
_{i}^{j}\\
&  =\left(  g_{j}\times\left[  s_{j}^{1},s_{j}^{2}\right]  _{E_{j}}\right)
\circ\varepsilon_{i}^{j}+\left(  \rho_{j}\left(  s_{j}^{1}\right)  \left(
g_{j}\right)  \times s_{j}^{2}\right)  \circ\varepsilon_{i}^{j}\\
&  =\left(  g_{j}\times\left[  s_{j}^{1},s_{j}^{2}\right]  _{E_{j}}+\left(
\rho_{j}\left(  s_{j}^{1}\right)  \right)  \left(  g_{j}\right)  \times
s_{j}^{2}\right)  \circ\varepsilon_{i}^{j}\text{.}%
\end{align*}
\qquad

3. Con la construcci\'{o}n de $\underrightarrow{\lim}[\;,\;]_{i}$,
 si $\rho_{i}$ es un morfismo de algebroides $\left(
E_{i},\pi_{i},M_{i},\rho_{i},[\;,\;]_{i}\right) \to
(TM_{i},M_{i},[\;,\;])$, entonces tenemos
\[
\underrightarrow{\lim}\rho_{i}(\underrightarrow{\lim}[\;,\;]_{i}%
)=[\underrightarrow{\lim}\rho_{i}(.),\underrightarrow{\lim}\rho_{i}(.)].
\]
Es f\'{a}cil probar que si cada corchete $[\;,\;]_{i}$ satisface la identidad de Jacobi, el l\'{\i}mite $\underrightarrow{\lim}[\;,\;]_{i}$ satisface tambi\'{e}n la identidad de Jacobi.
\end{proof}

Sea $\left(  E_{i},\tau_{i},M_{i},\rho_{i}\right)  _{i\in\mathbb{N}^{\ast}}$
una sucesi\'{o}n directa de algebroides de Lie donde $\left(  \left(
E_{i},\lambda_{i}^{j}\right)  \right)  _{i\in\mathbb{N}^{\ast},\ j\in
\mathbb{N}^{\ast},\ i\leq j}$ es la sucesi\'{o}n directa de fibrados
associada. Sea $\left(  P_{i},\nu_{i},M_{i}\right)  _{i\in\mathbb{N}^{\ast}}$
una sucesi\'{o}n directa de fibrados vectoriales y $\theta_{i}^{j}%
:P_{i}\longrightarrow P_{j}$ una aplication fibrada sobre $\varepsilon_{i}%
^{j}:M_{i}\longrightarrow M_{j}$. Puesto que $\lambda_{i}^{j}$ es un morfismo
entre los algebroides de Lie $\left(  E_{i},\tau_{i},M_{i},\rho_{i}\right)  $ y
$\left(  E_{j},\tau_{j},M_{j},\rho_{j}\right)$, entonces $T^{\lambda_{i}^{j}%
}\theta_{i}^{j}:\mathcal{T}^{E_{i}}P_{i}\longrightarrow\mathcal{T}^{E_{j}%
}P_{j}$ es un morfismo de prolongaciones de algebroides de Lie (cf.
Proposici\'{o}n \ref{P_MorphismExtensionLieAlgebroids}). Con la condici\'{o}n de
compatibilidad
\[
\rho_{P_{j}}\circ T^{\lambda_{i}^{j}}\theta_{i}^{j}=T\theta_{\iota}^{j}%
\circ\rho_{P_{i}}%
\]

$\left(  T^{E_{i}}P_{i},\tau_{P_{i}}^{E_{i}},P_{i},\rho_{P_{i}}\right)
_{i\in\mathbb{N}^{\ast}}$ es una sucesi\'{o}n directa de algebroides de Lie.

\bigskip
Utilizando el Teorema \ref{T_ConvenienLieAlgebroid} obtenemos el resultado siguiente:

\begin{theorem}
\label{T_ConvenientLieAlgebroidExtension} $\left(  \underrightarrow{\lim
}T^{E_{i}}P_{i},\underrightarrow{\lim}\tau_{P_{i}}^{E_{i}},\underrightarrow
{\lim}P_{i},\underrightarrow{\lim}\rho_{P_{i}}\right)  $ es un algebroide de
Lie conveniente.
\end{theorem}

\begin{example}
El oscilador arm\'{o}nico conveniente.-- Consideramos el caso del oscilador
arm\'{o}nico que es un sistema bihamiltoniano. Consideramos aqu\'{\i} el
l\'{\i}mite directo de prolongaciones de algebroides de Lie $T^{E_{n}}%
E_{n}^{\ast}$ donde $E_{n}=T\mathbb{R}^{n}$ y el ancla es el tensor de
Nijenhuis
\[
N_{n}=\left(
\begin{array}
[c]{ccccc}%
\dfrac{\left(  x^{1}\right)  ^{2}+\left(  y^{1}\right)  ^{2}}{2} & 0 &  &  &
\\
0 & \dfrac{\left(  x^{1}\right)  ^{2}+\left(  y^{1}\right)  ^{2}}{2} &  &  &
\\
&  & \ddots &  & \\
&  &  & \dfrac{\left(  x^{n}\right)  ^{2}+\left(  y^{n}\right)  ^{2}}{2} & 0\\
&  &  & 0 & \dfrac{\left(  x^{n}\right)  ^{2}+\left(  y^{n}\right)  ^{2}}{2}%
\end{array}
\right)  .
\]
Tenemos los corchetes siguientes: $\left[  \dfrac{\partial}{\partial x^{k}%
},\dfrac{\partial}{\partial y^{k}}\right]  _{N_{n}}=-y^{k}\dfrac{\partial
}{\partial x^{k}}+x^{k}\dfrac{\partial}{\partial y^{k}}$. \newline Si $\left(
x_{\ }^{i},\mu_{\alpha}\right)  _{1\leq i\leq n,1\leq\alpha\leq n}$ son
coordenadas sobre $T^{\ast}\mathbb{R}^{n}$ consideramos el hamiltoniano,
definido para $\left(  x_{\ }^{i},\mu_{\alpha}\right)  \neq\left(  0,0\right)
$, por%

\[
H_{n}:\left(  x_{\ }^{i},\mu_{\alpha}\right)  \mapsto\prod\limits_{i=1}^{n}%
\ln\left(  \left(  x^{i}\right)  ^{2}+\left(  \mu_{\alpha}\right)
^{2}\right)  .
\]
\newline Tenemos une sucesi\'{o}n proyectiva de funciones.\newline Obtenemos
las ecuaciones de Hamilton sobre cada $T^{\ast}\mathbb{R}^{n}$(cf. \cite{Mar2}
):%
\[
\left\{
\begin{array}
[c]{c}%
\dfrac{dx^{i}}{dt}=\rho_{\alpha}^{i}\dfrac{\partial H_{n}}{\partial\mu
_{\alpha}}\\
\dfrac{d\mu_{\alpha}}{dt}=-\rho_{\alpha}^{i}\dfrac{\partial H_{n}}{\partial
x^{i}}-\mu_{\gamma}C_{\alpha\beta}^{\gamma}\dfrac{\partial H_{n}}{\partial
\mu_{\beta}}%
\end{array}
\right.
\]
\newline\newline Estas ecuaciones se pueden simplificar de forma que:%
\[
\left\{
\begin{array}
[c]{c}%
\dfrac{dx^{i}}{dt}=\mu^{i}\\
\dfrac{d\mu^{\alpha}}{dt}=-x^{\alpha}%
\end{array}
\right.  \text{.}%
\]

\end{example}


\begin{thebibliography}{999999}                                                                                           


\bibitem[Bou]{Bou}N. Bourbaki, \textit{El\'{e}ments de Math\'{e}matiques},
Alg\`{e}bre, Chapitres 1 \`{a} 3, 2$^{\grave{e}me}$ \'{e}dition, Springer 2006

\bibitem[Cab]{Cab}P. Cabau, \textit{Strong projective limit of Banach Lie
algebroids}, Portugal. Math. (N.S.) Vol. 69, Fasc. 1, 2012, 1--21

\bibitem[CabPel]{CabPel}P. Cabau, F. Pelletier, \textit{Almost Lie structures
on an anchored Banach bundle}, Journal of Geometry and Physics 62 (2012) 2147--2169

\bibitem[Glo1]{Glo1}H. Gl\"{o}ckner, \textit{Direct limit of Lie groups and
manifolds}, J. Math. Kyoto Univ. (JMKYAZ) 43-1 (2003) 1--26

\bibitem[Glo2]{Glo2}H. Gl\"{o}ckner, \textit{Fundamentals of Direct Limit Lie
Theory}, Compositio Math. 141 (2005) 1551--1577

\bibitem[Glo3]{Glo3}H. Gl\"{o}ckner, \textit{Direct limits of
infinite-dimensional Lie groups compared to direct limits in related
categories}, Journal of Functional Analysis 245 (2007) 19--61

\bibitem[Han]{Han}V.L. Hansen, \textit{Some Theorems on Direct Limit of
Expanding Sequences of Manifolds}, Math.\ Scand.29 (1971) 5--36

\bibitem[HigMac]{HigMac}P.J. Higgins, K.C.H. Mackenzie, \textit{Algebraic
constructions in the category of Lie algebroids}, J. Algebra, 129 (1990) 194--230

\bibitem[Igl]{Igl}D. Iglesias Ponte, \textit{Variedades de Poisson, grupoids y
algebroides de Lie,} Actas del XI Congreso Dr. Antonio A. R. Monteiro (2011) 35--59

\bibitem[KriMic]{KriMic}A. Kriegel, P.W. Michor, \textit{The convenient
Setting of Global Analysis }(AMS Mathematical Surveys and Monographs)
\textbf{53} 1997

\bibitem[Lan]{Lan}S. Lang, \textit{Differential and Riemannian Manifolds},
Graduate Texts in Mathematics, 160, Springer, New York 1995

\bibitem[LMM]{LMM}M. de Le\'{o}n, J.C. Marrero, E. Mart\'{\i}nez,
\textit{Lagrangian submanifolds and dynamics on Lie algebroids}, J. Phys. A:
Math. Gen., 38 (2005), R241-R308.

\bibitem[MagMor]{MagMor}F. Magri, C. Morosi,\textit{\ A geometrical
characterization of integrable hamiltonian systems through the theory of
Poisson-Nijenhuis manifolds, }Quaderno S 19, Universit\`{a} degli studi di
Milano, 1984

\bibitem[Mar1]{Mar1}E. Mart\'{\i}nez, \textit{Classical field theory on Lie
algebroids: multisymplectic formalism}, math.DG/0411352

\bibitem[Mar2]{Mar2}E. Mart\'{\i}nez, \textit{Lie algebroids in Classical
Mechanics and Optimal Control}, Symmetry, Integrability and Geometry: Methods
and Applications SIGMA 3 (2007), 050

\bibitem[Nie]{Nie}D. de las Nieves Sosa Mart\'{\i}n, \textit{Afgebroides de
Lie y Mec\'{a}nica Geom\'{e}trica}, Tesis, Universidad de La Laguna, 2008

\bibitem[Pra]{Pra}J. Pradines, \textit{Th\'{e}orie de Lie pour les
groupo\"{\i}des diff\'{e}rentiables ; relations entre propri\'{e}t\'{e}s
locales et globales}, C.R. Acad. Sci. Paris 263 (1966) 907-- 910.

\bibitem[SurCab]{SurCab}A. Suri, P. Cabau, \textit{Geometric structure for the
tangent bundle of direct limit manifolds},  Differential Geometry - Dynamical Systems, Vol.16 (2014) 239—247

\bibitem[Wei]{Wei}A. Weinstein, \textit{Lagrangian Mechanics and groupoids},
Fields Inst. Comm., 7 (1996), 207--231
\end{thebibliography}
\end{document}